\documentclass[11pt, oneside]{amsart}

\usepackage[english]{babel}
\usepackage{mathtools}
\usepackage{wrapfig}
\usepackage{tikz}
\usepackage{tikz-cd}
\usepackage{quiver}
\usepackage{comment}
\usepackage{mathrsfs}
\usepackage{tabularx, hyperref}
\usepackage{amssymb} \usepackage{amsfonts} \usepackage{amsmath}
\usepackage{amsthm} \usepackage{epsfig, subfig}
\usepackage{amscd, amsxtra, latexsym}
\usepackage{epsfig,  graphicx, psfrag}
\usepackage[all]{xy}
\usepackage{caption}
\usepackage{enumerate}
\usepackage{color}

\usepackage{setspace}
\usepackage[colorinlistoftodos]{todonotes}
\newtheorem{lemma}{Lemma}[section]
\newtheorem{thm}[lemma]{Theorem}
\newtheorem{prop}[lemma]{Proposition}

\newtheorem*{prop*}{Proposition}

\newtheorem{theorem}{Theorem}

\newtheorem{rthm}{Theorem}
\newenvironment{reptheorem}[1]
{\renewcommand\therthm{\ref*{#1}}\rthm}
  {\endrthm}

\theoremstyle{definition}

\newtheorem{defn}[lemma]{Definition}

\newtheorem{rem}[lemma]{Remark}

\theoremstyle{definition}
\newtheorem*{claim}{Claim}

\definecolor{darkgreen}{cmyk}{1,0,1,.2}

\DeclareMathOperator{\sym}{Sym}

\DeclareMathOperator{\aut}{Aut}
\DeclareMathOperator{\alt}{alt}
\DeclareMathOperator{\id}{id}

\newcommand{\N}{\ensuremath {\mathbb{N}}}
\newcommand{\R} {\ensuremath {\mathbb{R}}}

\newcommand{\Z} {\ensuremath {\mathbb{Z}}}
\renewcommand{\phi}{\varphi}

\begin{document}

\title[Bounded cohomology of groups acting on trees]{Bounded cohomology of groups acting on trees with almost prescribed local actions}

\author{Giuseppe Bargagnati}
\address{Dipartimento di Matematica, Universit\`a di Bologna, Italy}
\email{giuseppe.bargagnati2@unibo.it}
\author{Elena Bogliolo}
\address{Dipartimento di Matematica, Universit\`a di Pisa, Italy}
\email{elena.bogliolo@phd.unipi.it}

\begin{abstract}
    We prove the vanishing of bounded cohomology of the groups acting on trees with almost prescribed local actions $G(F, F')$, where $F<F'$ are finite permutation groups such that $F'$ is 2-transitive. By contrast, when $F'$ is not 2-transitive, we prove that the second bounded cohomology with real coefficients of the groups $G(F, F')$ is infinite dimensional. 
\end{abstract}

\maketitle

\section{Introduction}

Despite being widely applied, bounded cohomology is in general very hard to compute. One of the very first results in the field was the vanishing of bounded cohomology of amenable groups in every (positive) degree \cite{Joh72}.
The groups for which the bounded cohomology vanishes in all positive degrees are said to be \emph{boundedly acyclic}. The first non-amenable example of a boundedly acyclic group was found by Matsumoto and Morita, who proved that the group of compactly supported homemomorphisms of the Euclidean $n$-space is boundedly acyclic \cite{MM85}. 
In recent years, various classes of groups have been proved to be boundedly acyclic: among them there are lamplighter groups \cite{Monod2022}, binate groups \cite{FLM22}, groups which admit commuting cyclic conjugates \cite{CFFLM24, CFFLM24bis} and the group of \emph{all} homeomorphisms of $\R^n$ \cite{FFMNK24}. 
A standard way to prove that a non-amenable group is boundedly acyclic is to make use of the so-called displacement techniques. These techniques require copies of subgroups that commute with each other.

While many results have been obtained about the vanishing of bounded cohomology, the non-vanishing in infinitely many degrees is by now more elusive.
However, in degree two, the situation is different. In fact, it is sufficient to find a homogeneous quasimorphism which is not a homomorphism to prove that the second bounded cohomology with real coefficients of the group does not vanish. Many results have been obtained on the non-vanishing of second bounded cohomology studying quasimorphisms \cite{EF97, Fu98, Fu00, BF02}. 

In this article, we study the bounded cohomology of groups $G(F, F')$ acting on trees with almost prescribed local actions, where $F<F'$ are finite permutation groups (see Section \ref{sec:groups} for the definition). Such groups were defined by Adrien Le Boudec \cite{LB15}, and provided the first known examples of non-$C^*$-simple groups with trivial amenable radical \cite{lebouinventiones}, answering a long-standing open question \cite{BD00}. 

The groups acting on trees with almost prescribed local actions, endowed with the topology described in Section \ref{sec:groups}, are locally compact groups, so we study their continuous bounded cohomology (see Section \ref{sec:bc} for the definition). Our main result is the following.

\begin{theorem}\label{principale}
    Let $d\geq 3$ be a natural number, let $\Omega=\{1, \ldots, d\}$ and let $F<F'<\sym(\Omega)$ be permutation groups such that $F'$ preserves the orbits of $F$. The following hold:
    \begin{itemize}
        \item[(1)]\label{item1} if $F'$ acts 2-transitively on $\Omega$, then the group $G(F, F')$ is boundedly acyclic;
        \item[(2)]\label{item2} otherwise, $H^2_{\textup{cb}}(G(F, F'))$ is infinite dimensional.
    \end{itemize}
\end{theorem}

The displacement techniques usually adopted to prove that a non-amenable group is boundedly acyclic are generally applied to groups which contain direct product subgroups of infinite cardinality. Therefore, the fact that groups \(G(F,F')\) do not contain finitely generated direct products of infinite cardinality (Lemma \ref{directprod}) witnesses the difficulty in adapting these techniques in the context of groups acting on trees with almost prescribed local actions. 
For this reason, to prove item (1) of Theorem \ref{principale}, we adopt a different strategy, which exploits the aligned complex introduced by Bucher and Monod \cite{BM17}. This technique has found other remarkable applications to bounded cohomology in recent years \cite{BM17bis, AB22, Mar22}, and has been lately generalized by Monod to groups acting properly and strongly transitive on buildings \cite{Mon24}. However, in the situation studied by Bucher and Monod, the groups act transitively on the boundary of a locally finite tree, whereas the group $G(F, F')$ may not act transitively on the boundary of said regular tree (for example, if $F$ acts freely on $\Omega$ then \(G(F,F')\) is countable \cite[Definition 3.1]{LB15}). In order to circumvent this issue, we prove a result of independent interest about the vanishing of the bounded cohomology of groups acting amenably on locally finite trees.

\begin{theorem}\label{secondario}
    Let $G$ be a locally compact group acting amenably and without inversions on a locally finite (regular) tree $\mathcal{T}_d$.\ Let us suppose that 
    there exists a line $L$ in $T$ with the following properties:
    \begin{itemize}
        \item for every segment of any (fixed) finite length $[x,y]$ there exists $g \in G$ such that $[gx, gy]\subset L$;
        \item the set-wise stabilizer of $L$ acts transitively on the geometric edges of $L$.
    \end{itemize}
    Then $H^n_{\textup{cb}}(G;\R)=0$ for each $n\geq1$.
\end{theorem}

Also the previous theorem admits an interpretation in terms of buildings, which was pointed out to us by Nicolas Monod, for which we refer to Remark \ref{remarkbuildings}.

On the other hand, to prove item (2) of Theorem \ref{principale}, we use a result by Iozzi, Pagliantini and Sisto \cite[Theorem 1]{IPS20} which gives sufficient conditions for the existence of an infinite family of linearly independent coclasses in the second bounded cohomology with real coefficients of groups acting minimally on locally finite trees.

\subsection*{Plan of the paper} In Section \ref{sec:groups} we define the groups $G(F, F')$ and state some of their main properties. In Section \ref{sec:bc} we recall the definition of bounded cohomology of groups and quasimorphisms. In Section \ref{sec:aligned} we define the aligned complex introduced by Bucher and Monod \cite{BM17}. Finally, Section \ref{sec:proofs} is devoted to the proof of Theorem \ref{principale}.

\subsection*{Acknowledgments} We thank Francesco Fournier-Facio for suggesting us this problem and for many helpful comments and suggestions. We thank Adrien Le Boudec and Nicolas Monod for helpful comments. The second author thanks her advisor Marco Moraschini for his guidance throughout this project and for useful discussions. This paper was funded by the European Union - NextGenerationEU under the National Recovery and Resilience Plan (PNRR) - Mission 4 Education and research - Component 2 From research to business - Investment 1.1 Notice Prin 2022 - DD N. 104 del 2/2/2022, from title “Geometry and topolgy of manifolds”, proposal code 2022NMPLT8 - CUP J53D23003820001.

\section{Groups acting on trees with almost prescribed local action}\label{sec:groups}
In this section we recall the definition of the groups of automorphisms of a locally finite regular tree with \emph{almost prescribed local action} introduced by Le Boudec \cite{LB15}.\\
Here and henceforth, we consider simplicial group actions on a tree, and we may identify the action of a group on a tree with the action on the vertices of the tree. From now on, let $d\geq 3$ be a natural number and let $\Omega=\{1, \ldots, d\}$. Let us denote with $T=\mathcal{T}_d$ a regular tree of degree $d$ and with \(V(T)\) and \(E(T)\) the set of vertices 
and the set of non-oriented edges of $T$, respectively. Let us denote with $\aut(T)$ the group of automorphisms of $T$ endowed with the pointwise-convergence topology. Let $F<F'<\sym(\Omega)$ be permutation groups such that \(F'\) preserves the orbits of $F$. 
Let us fix a coloring 
\[c \colon E(T) \to \Omega\] 
such that for every vertex
\(v \in V(T)\) the map \(c\) restricts to a bijection \(c_v\) from the set \(E(v)\) of edges
containing \(v\) to \(\Omega\). In this setting \(c(e)\) will denote the color of the edge \(e\).
Given an automorphism \(g \in \aut(T)\) and a vertex \(v \in V(T)\), we denote with
\(g_v \colon E(v) \to  E(g\cdot v)\), the induced bijection on the edges. The combination of \(g_v\) with the coloring gives rise to a permutation \(\sigma(g, v) \in \sym(\Omega)\). We call \(\sigma(g, v)\) the \emph{local
permutation of g at the vertex \(v\)} and we define it
by \[\sigma(g, v) \coloneqq c_{gv} \circ g_v \circ c_v^{-1}.\] 
We recall the definitions of the groups \(U(F)\) introduced by Burger and Mozes \cite{BM00} and of the groups $G(F, F')$ introduced by Le Boudec \cite{LB15}.

\begin{defn}[Universal Burger-Mozes groups]
    Let \(F \leq \sym(\Omega)\) be a permutation group. The \emph{universal Burger-Mozes group} \(U(F )\) is the subgroup of \(\aut{(T)}\) whose
local action is prescribed by \(F\), that is,
 \[U(F)\coloneqq\{g \in \aut(T) \colon \sigma(g, v) \in F \mbox{ for all } v \in V(T)\}.\]
\end{defn}

Note that, the universal Burger-Mozes groups are closed subgroups of \(\aut(T)\) with respect to the pointwise-convergence topology, and for every \(F \leq F'\), the group \(U(F)\) is a subgroup of \(U(F')\). Moreover, the group \(U(F)\) is discrete if and only if the permutation group \(F\) acts freely on \(\{1,\dots,n\}\) \cite[Section 3.1]{BM00}. 

\begin{defn}[Groups acting on trees with almost prescribed local actions]
    Let \(F\leq F'\) be subgroups of the symmetric group \(\sym(\Omega)\) such that \(F'\) preserves the orbits of \(F\).
    We set \[G(F)\coloneqq \{g \in \aut(T) \colon \sigma(g, v) \in F \mbox{ for all but finitely many } v \in V(T) \}.\]
    The group \(G(F,F')\) is defined as
\[G(F,F')\coloneqq G(F)\cap U(F').\] 
\end{defn}

The group $G(F, F')$ can be endowed with a natural group topology, which is defined by requiring that the inclusion $U(F) \hookrightarrow G(F, F')$ is continuous and open, where $U(F)$ is endowed with the induced topology from $\aut(T)$. With this topology, $G(F, F')$ is a compactly generated, totally disconnected, locally compact group \cite[Section 1.1]{LB15}. A totally disconnected, locally compact group is called \emph{locally elliptic} if it is a directed union of compact open subgroups.

\begin{rem}\label{elliptic}
    Since the point stabilizers in $G(F)$ are locally elliptic groups, so are the point stabilizers in $G(F, F')$. In fact, for every vertex \(v\in T\), the point stabilizer \(G(F,F')_v\) is the intersection of the closed subgroup \(U(F')_v\) and the locally elliptic subgroup \(G(F)_v\) \cite[Section 3.1]{LB15}. We also notice that locally elliptic groups are amenable, since they are a directed union of compact (hence, amenable) subgroups.
\end{rem}

\begin{rem}\label{notamen}
    We note that $U(\id)$ is a subgroup of $G(F, F')$ for any choice of permutation groups $F$ and $F'$ which satisfy the requirements of the definition. The group $U(\id)$ is isomorphic to the free product of $d\geq 3$ copies of $\mathbb{Z}/2$, where $d$ is the degree of $T$. Hence, the groups $G(F, F')$ are not amenable.
\end{rem}

The following lemma was pointed out to us by Francesco Fournier-Facio.

\begin{lemma}\label{directprod}
    Let us assume that \(F\) acts freely. Then every finitely generated direct product subgroup of $G(F, F')$ is finite.
\end{lemma}

\begin{proof}
    It is sufficient to prove that when $H<G(F, F')$ is a finitely generated direct product, it fixes a point. Indeed, since the point-stabilizers in $G(F, F')$ are locally elliptic and \(G(F,F')\) is discrete when \(F\) acts freely \cite[Section 3.1]{lebouinventiones}, then the point stabilizers are also locally finite. This would imply that $H$ is finite.\\
    Let $H\cong A \times B$, with $A \ncong 1$ and $B \ncong 1$ and let us assume by contradiction that there exists a loxodromic element $a \in A$ with axis $r$. Then we have a homomorphism $\phi:<a>\times B \rightarrow \Z$ given by the translation length along $r$, whose kernel is either trivial or consists of elliptic elements \cite[Lemma 5.3]{button2019groupsactingfaithfullytrees}. 
    In the first case, since the restriction of $\phi$ to $<a>\times \id_{B}$ is an isomorphism, then $B$ would have to be trivial. Otherwise, $B$ consists of elliptic elements which fix $r$ pointwise. For each $b \in B$ there exists a finite subset of vertices of $r$ such that $b$ moves the neighbors of these vertices, let us call this the \emph{support} of $b$. Then the support of $a^{-1}ba$ consists of a unitary translation (along the axis $r$) of the support of $b$.
    But $a$ and $b$ commute, so this implies that $b=\id$. This shows that $A$ and $B$ consist only of elliptic elements, and since they are finitely generated it follows that both $A$ and $B$ admit a global fixed point in $T$ \cite[Corollary 2.5]{GGT18}.
    Since $B$ is elliptic and acts on the tree of fixed points of $A$, it has a limited orbit in this tree, and hence it has a fixed point \cite[Corollary 2.5]{GGT18}, which is in fact a fixed point for the direct product $A \times B$.
\end{proof}

We state the following classical lemma, which will be useful later on.

\begin{lemma}[{\cite[Lemma 7.1.1]{Robinson1996}}]\label{stabiltrans}
    Let $n\geq 3$ be a natural number, $\Omega=\{1, \ldots, n\}$ and let $F< \sym(\Omega)$ be a permutation group. Then $F$ is 2-transitive if and only if the stabilizer of each point $x\in \Omega$ in $F$ acts transitively on $\Omega \setminus\{x\}$.
\end{lemma}

\section{Bounded cohomology of groups}\label{sec:bc}

Let $G$ be a locally compact group. Let us denote by $C^n_{\textup{cb}}(G)$ the space of continuous bounded maps from $G^{n+1}$ to $\R$, where \(\R\) is endowed with the Euclidean topology. We have an action of $G$ on such functions defined by 
\[(g\cdot f)(g_0,\dots,g_n)=f(g^{-1}g_0,\dots, g^{-1}g_n).\]
Let us denote by $C^n_{\textup{cb}}(G)^G$ the subset of $G$-invariant continuous bounded functions from $G^{n+1}$ to $\R$. We call \emph{real bounded cohomology} of $G$ the cohomology of the following complex of vector spaces:
\[
	0 \rightarrow C^0_{\textup{cb}}(G)^G \xrightarrow{\delta^0} C^1_{\textup{cb}}(G)^G \xrightarrow{\delta^1} C^2_{\textup{cb}}(G)^G \xrightarrow{\delta^2}\cdots\,,
\]
where the differential maps  are defined by \[(\delta^{n-1} f)(g_0,\dots,g_n) = \sum_{i=0}^n (-1)^i f(\dots,\hat{g_i},\dots).\] 
In symbols, the $n$-th continuous bounded cohomology of $G$ with $\R$-coefficients is defined as  \[H^n_{\textup{cb}}(G; \R)\coloneqq \frac{\ker(\delta^n)}{\delta^{n-1}(C_{\textup{cb}}^{n-1}(G)^G)}.\]

\begin{defn}
    A group $G$ is \emph{boundedly acyclic} if $H^n_{\textup{cb}}(G; \R)=0$ for every $n\geq1$.
\end{defn}

The inclusion of bounded cochains in the usual cochains induces a map between the continuous bounded cohomology and the ordinary continuous cohomology of \(G\) called the \emph{comparison map}:
\[c^n\colon H_{\textup{cb}}^n(G;\R) \to H^n_{\textup{c}}(G;\R).\] This map is in general neither injective nor surjective \cite{Mon01, frigerio:book}.

\subsection{Quasimorphisms}
We recall the relation between second bounded cohomology and quasimorphisms \cite[Section 2.3]{frigerio:book}. Moreover, we define the class of \emph{median quasimorphisms} for groups acting on trees.

\begin{defn}
    Let \(G\) be a group. A quasimorphism is a function \(\phi\colon G\to~\R\) for which there is a least finite constant \(D(\phi) \geq 0\) such that
\[|\phi(ab) -\phi(a)- \phi(b)| \leq D(\phi)\]
for all \(a, b \in G\). We denote by \(Q(G)\) the vector space of quasimorphisms of \(G\). If \(G\) is a topological group we denote with \(Q_{\textup{c}}(G)\) the subspace of continuous quasimorphisms.\\
A quasimorphism \(\phi\colon G\to \R\) is \emph{homogeneous} if \(\phi(g^n)=n\phi(g)\) for every \(g\in G\) and for every \(n\in \Z\).
\end{defn}

Trivial examples of continuous quasimorphisms are continuous bounded functions, which we denote by $C_{\textup{cb}}^0(G;\R)$, and continuous homomorphisms from $G$ to $\R$, which are also homogeneous, and which we denote by $\textup{Hom}_\textup{c}(G)$. Given any quasimorphism \(\phi\in Q(G)\) we can obtain a homogeneous quasimorphism at finite distance from it via the limit:
\[\phi^h(g)\coloneqq \lim_{n\to \infty}\frac{\phi(g^n)}{n}.\]
This limit exists and defines a quasimorphism whose distance from $\phi$ is bounded by the defect $D(\phi)$ \cite[Lemma 2.21]{calegari2009scl}. 

Quasimorphisms allow us to compute the second bounded cohomology of groups thanks to the following result.

\begin{prop}[{\cite[Corollary 13.3.2]{Mon01}}]
    Let \(G\) be a locally compact group. There is a canonical
isomorphism of vector spaces
\[ \ker(c^2) \cong \frac{Q_{\textup{c}}(G)}{C_{\textup{b}}(G;\R)\bigoplus \textup{Hom}_{\textup{c}}(G)}\ .\]
\end{prop}

Let now $G$ be a locally compact group which acts (simplicially) on a locally finite tree. On such groups, we define median quasimorphisms as follows.

\begin{defn}
    Let \(s\) be an oriented geodesic segment of a tree \(T\). For any vertex \(v\in V(T)\), we define the \emph{median quasimorphism} 
\[f_{s,v} \colon G \to  \Z,\] 
\begin{eqnarray*}f_{s,v}(g)=\#\{\mbox{oriented \(G\)-translates of \(s\) in \([v, g\cdot v]\)}\}\\-\#\{\mbox{oriented \(G\)-translates of \(s\) in \([g\cdot v, v]\)}\}.
\end{eqnarray*}
  
\end{defn}

In order to represent classes in the continuous second bounded cohomology, homogeneous median quasimorphisms must be continuous, as observed in the following remark.

\begin{rem}[Homogenized median quasimorphisms are continuous]\label{continuous}
    The homogenization of median quasimorphisms can be described in the same way as the one of Brooks quasimorphisms \cite[Theorem 3.37]{FF20}, i.e., counting the number of $G$-translates on the circular path \([v, g\cdot v]/_{g\cdot v=v}\). This shows that the homogenization of median quasimorphisms has values in \(\Z\).
    Since every homogeneous quasimorphism from a locally compact group to the integers is continuous \cite[Theorem A.1]{CF10}, we get the continuity of homogenizations of median quasimorphisms. 
\end{rem}

 The following result by Iozzi, Pagliantini and Sisto will play a crucial role in the proof of the second item of Theorem \ref{principale}. We state it for continuous bounded cohomology of locally compact groups. Indeed, in this case the homogeneization of the median quasimorphisms, which appear in the proof of Iozzi Pagliantini and Sisto, are continuous by Remark \ref{continuous}, and they provide an infinite number of linearly independent continuous bounded cohomology classes.

\begin{thm}{\cite[Theorem 1]{IPS20}}\label{thmIozzi}
    Let $G$ be a locally compact group that acts minimally on a regular locally finite tree $\mathcal{T}$ such that every vertex has valence greater or equal than 2. Then one of the following holds:
    \begin{itemize}
        \item[(a)] $G$ fixes a point in the boundary $\partial \mathcal{T}$ of $\mathcal{T}$;
        \item[(b)] $G$ acts transitively on the sets of geodesics of any fixed length starting at any vertex (i.e., the stabilizer of any fixed vertex  acts in such a way);
        \item[(c)] There exists an infinite family of linearly independent coclasses in $H^2_{\textup{cb}}(G, \R)$ given by median quasimorphisms.
    \end{itemize}
\end{thm}

We omit the mutual exclusivity of the three propositions in the statement of Theorem \ref{thmIozzi} referring to the observation in Hofmann's master thesis \cite{Franziska}.

\section{Aligned chains and aligned complex}\label{sec:aligned}

In this section we recall the definition of the \emph{complex of aligned chains}, as introduced by Bucher and Monod \cite{BM17}.\\
Let \(\mathcal{C}_*(T, \R)\) denote the module of \(\R\)-valued
alternating \(n\)-chains on the tree \(T\), that is the quotient of the free \(\R\)-module on \(T^{n+1}\) by the equivalence relation obtained identifying an \((n+1)\)-tuple with its permutations with sign (i.e., \((x_0, \dots , x_n) \sim \textup{sgn}(\sigma)(x_{\sigma(0)}, \dots , x_{\sigma (n)})\), for all \(\sigma \in \sym(\Omega)\)). We obtain the resolution
\[0 \leftarrow \R \leftarrow \mathcal{C}_0(T, \R) \leftarrow \mathcal{C}_1(T, \R) \leftarrow \mathcal{C}_2(T, \R) \leftarrow\cdots \]
with boundary maps \(\delta_{n-1}:\mathcal{C}_n(T,\R)\to \mathcal{C}_{n-1}(T,\R)\) defined by \[\delta_{n-1}((x_0,\dots,x_n)) \coloneqq
\sum_{j=0}^n
(-1)^j(x_0,\dots,\hat{x}_j,\dots, x_n),\] where we denote with \(\hat{x}_j\) the omitted variable, and \(\delta_{-1}\colon \mathcal{C}_0(T,\R)\to \R\)
 is given by the summation. It is easy to check that this is an exact sequence. 

\begin{defn}[aligned chains]
    The complex of (alternating) \emph{aligned chains} is the subcomplex \(\mathcal{A}_*(T, \R)\) of
\(\mathcal{C}_*(T, \R)\) spanned by those \((n + 1)\)-tuples that are contained in some geodesic segment in \(T\).
\end{defn}

Unless \(T\) is a linear tree, \(\mathcal{A}_n(T, \R)\) is a proper subgroup of \(\mathcal{C}_n(T, \R)\) provided that \(n \geq 2\).\\
We endow \(\mathcal{C}_*(T, \R)\) and
\(\mathcal{A}_*(T, \R)\) with the quotient norm induced by the \(\ell^1\)-norm on
\(\ell^1(T^{n+1})\). One can check that the boundary map is a 
bounded operator with respect to this norm. We can take the topological dual of the normed spaces and obtain the cochain complex

\[0\to \R \to \ell^{\infty}_{\mathcal{A}}(T,\R)\to \ell^{\infty}_{\mathcal{A}}(T^2,\R)\to \ell^{\infty}_{\mathcal{A}}(T^3,\R)\to \cdots \]
where \(\ell^{\infty}_{\mathcal{A}}(T^n,\R)\) is the space of bounded alternating functions on \(T^{n}\) that are supported on aligned tuples. As proved by Bucher and Monod, this is a resolution for \(\R\) \cite[Corollary 8]{BM17}.

\section{Proof of Theorem \ref{principale} and Theorem \ref{secondario}}\label{sec:proofs}

This section is devoted to the proof of Theorem \ref{principale}. The first part of it will follow as a corollary of Theorem \ref{secondario}, which we state again for convenience of the reader.

\begin{reptheorem}{secondario}\label{bcsegmfinleng}
    Let $G$ be a locally compact group acting amenably and without inversions on a locally finite (regular) tree $\mathcal{T}_d$.\ Let us suppose that 
    there exists a geodesic line $L$ in $T$ with the following properties:
    \begin{itemize}
        \item [(i)] for every segment of any (fixed) finite length $[x,y]$ there exists $g \in G$ such that $[gx, gy]\subset L$;
        \item [(ii)]the set-wise stabilizer of $L$ acts transitively on the geometric edges of $L$.
    \end{itemize}
    Then $H^n_{\textup{cb}}(G;\R)=0$ for each $n\geq1$.
\end{reptheorem}

\begin{proof}
    
We follow the blueprint of the proof of Bucher and Monod \cite[Theorem 1]{BM17}, but we need some changes. 

Let us identify $T$ with its set of vertices. We equip $T$ with a probability measure of full support, and we observe that $T$ is an amenable regular $G$-space in the sense of Monod \cite[Definition 2.1.1]{Mon01}. Indeed the action of $G$ on $T$ is measure-class preserving, and amenable by assumption. The space \(T^n\), endowed with the product probability, is also an amenable regular \(G\)-space \cite[Example 2.1.2, (ii)]{Mon01} and since the action of $G$ on \(T^n\) preserves aligned tuples, we have that also the subspace of aligned tuples is an amenable regular $G$-space.

 Using the amenability of the action on aligned tuples and a result by Monod \cite[Theorem 7.5.3]{Mon01} we have that the resolution
    \[0 \to \ell^{\infty}_{\mathcal{A}}(T, \R)^G \to \ell^{\infty}_{\mathcal{A}}(T^2, \R)^G \to \ell^{\infty}_{\mathcal{A}}(T^3, \R)^G \to \ldots\]
    computes the bounded cohomology of $G$.
    
    Let now $L$ be the geodesic line of the statement. We show that the set-wise stabilizer \(H\) of \(L\) is amenable. Given a vertex $v \in T$, let us denote by $G_v$ the stabilizer of $v$ in $G$. Since $G$ acts without inversions,
    we have that an element $g \in G$ can act on (the vertices of) $L$ either via translations or inverting a line rotating around a vertex (or a combination of the two). This gives a homomorphism $\phi:H \to D_{\infty}$, where \(D_{\infty}\) denotes the infinite dihedral group. We notice that $D_{\infty}$ is an amenable group (in fact, it is an extension of amenable groups), and so also the image of $\phi$ is amenable.
    The kernel of $\phi$ is equal to the intersection of all the point-wise stabilizers in $G$ of all the vertices of $L$, which are amenable since the action of $G$ on $T$ is amenable. We have the short exact sequence
    \[1\to \bigcap_{x\in L} G_x \to H \to \mbox{Im}(\phi) \to 1.\]
    Hence the stabilizer $H$ is an extension of amenable groups, so it is amenable.\\
    In order to conclude, it suffices to prove the following claim.
    \begin{claim}
        The map
        \[\psi \colon \ell^{\infty}_{\mathcal{A}}(T^{n+1}, \R)^G\to \ell^{\infty}_{\alt}(L^{n+1}, \R)^H\]
        induced by the restriction to $L$ is an isomorphism for every $n\geq 0$.
    \end{claim}
\begin{proof}[Proof of the claim]
    Each aligned tuple is contained in a segment of finite length in some geodesic $L'$. Hence, by assumption for each aligned tuple we can find an element \(g\in G\)
    that sends the given tuple into a new tuple which lies entirely in $L$. This proves that the map is injective.
    For the surjectivity, we take an element $f \in \ell^{\infty}_{\alt}(L^{n+1}, \R)^H$ and we proceed to extend it to a function $\Tilde{f}\in \ell^{\infty}_{\mathcal{A}}(T^{n+1}, \R)^G$. We pick an aligned tuple $\Bar{x}=(x_0, \ldots, x_n)$ in $T$, and we may assume that it lies in the geodesic segment $[x_0, x_n]$. By assumption, there exists $g \in G$ such that $g([x_0, x_n]):=[x_0', x_n']\subset L$, i.e., such that $g\Bar{x}$ lies on $L$. Now we claim that if $g'$ is any other element with the property that $g'\Bar{x}$ lies in $L$, then $f(g\Bar{x})=f(g'\Bar{x})$. The claim implies that we can set $\Tilde{f}(\Bar{x})=f(g\Bar{x})$ in a well-posed and $G$-invariant way.

    In order to prove the claim, we look for an element $h \in H$ such that $hg\Bar{x}=g'\Bar{x}$. Since $\Bar{x}$ is contained in the segment $[x_0, x_n]$, it suffices to find an $h$ such that $hgx_0=g'x_0$ and $hgx_n=g'x_n$. Let $y$ be the first vertex after $x_0$ in the segment $[x_0, x_n]$. 
    Since $H$ acts transitively on geometric edges (i.e.\ unordered pairs of vertices at distance one) of $L$ by assumption \((ii)\), there exists $h \in H$ such that $h\{gx_0, gy\}=\{g'x_0, g'y\}$.
    If we exclude that $hgx_0=g'y$, then it follows that $hgx_0=g'x_0$, which would imply that $hgx_n=g'x_n$. Since $y$ is adjacent to $x_0$ and $G$ acts without inversions, the case $hgx_0=g'y$ cannot happen, and this proves the claim.
\end{proof}
This concludes the proof of the theorem as the cochain complex \(\ell^{\infty}_{\alt}(L^{\bullet};\R)^H\) computes the continuous bounded cohomology of \(H\) and so \[H_{\textup{cb}}^n(G;\R)\cong H_{\textup{cb}}^n(H;\R)=0\] for all \(n\geq 1\) by amenability of \(H\) \cite{Joh72}.
\end{proof}

\begin{rem}\label{remarkbuildings}
    Theorem \ref{secondario} can also be formulated in the context of buildings: in fact, the set of the $G$-translates of the line $L$ from the statement of the theorem is a system of apartments for the tree $T$ on which the group $G$ acts strongly transitively (see, e.g., \cite{Tits74} or \cite[Section 4.1]{AB08} for the basic definitions). In this perspective, Theorem \ref{secondario} can be seen as an analogue of a recent theorem from Monod \cite[Theorem B]{Mon24}. However, in our statement the assumption that the action is proper is replaced by requiring that the action is amenable. This is indeed needed: if $F$ is a proper subgroup of $F'$, the groups with almost prescribed local actions $G(F, F')$ cannot act properly on a tree \cite[Section 1.6]{LB15}.
\end{rem}

Using the theorem above to obtain the first statement, and Theorem \ref{thmIozzi} for the second statement, we can finally prove the main result of this paper, that we state again for convenience of the reader.

\begin{reptheorem}{principale}
    Let $d\geq 3$ be a natural number, let $\Omega=\{1, \ldots, d\}$ and let $F<F'<\sym(\Omega)$ be permutation groups such that $F'$ preserves the orbits of $F$. The following hold:
    \begin{itemize}
        \item[(1)]\label{item1} if $F'$ acts 2-transitively on $\Omega$, then the group $G(F, F')$ is boundedly acyclic;
        \item[(2)]\label{item2} otherwise, $H^2_{\textup{cb}}(G(F, F'))$ is infinite dimensional.
    \end{itemize}
\end{reptheorem}

\begin{proof}[Proof of Theorem \ref{principale}, item (1)]
First, we notice that $G(F, F')$ acts transitively on the vertices of $T$, since its subgroup $U(\id)$ does \cite[Section 3.1]{LB15}.

Moreover, since $F'$ acts 2-transitively on $\Omega$, it holds that $G(F, F')$ also acts transitively on segments of finite length. Indeed, given two segments \([x_1,x_n]\) and \([x_1',x'_n]\) of length \(n\in \N\), by vertex transitivity there exists an element $g\in G(F, F')$ such that $g(x_1)=x_1'$. Now, for each vertex \(v\in[x_1,x_n]\), the local action $\sigma(g,v)$ can be given by a permutation belonging to $F'$. Let $x_2$ and $x_2'$ be the vertices of the segments $[x_1, x_n]$ and $[x_1', x_n']$ adjacent to $x_1$ and $x_1'$, respectively. By 2-transitivity of the action of \(F'\) on the edges we can send the edge \((x_1,x_2)\) to the edge \((x_1',x_2')\). In this way, \(x_2\) is mapped to \(x'_2\) and by 2-transitivity of the action of \(F'\) on the edges which have \(x_2\) as one of the endpoints we can map the edge \((x_2,x_3)\) to \((x_2',x_3')\) while sending \((x_1,x_2)\) to the edge \((x_1',x_2')\). We argue in the same way for the remaining vertices of the segment and extend this map to an element of \(G(F,F')\) using the transitivity of \(F\) to define the local actions around the remaining vertices of the tree. In particular, we can send each segment into any given geodesic line $L$ with an even displacement length, i.e., such that the distance between one (hence, any) vertex and its image is even.

The action of \(G(F,F')\) on \(T\) endowed with a probability measure of full-support is amenable. Indeed, by Remark \ref{elliptic}, the stabilizers of the action of \(G(F,F')\) on \(T\) are amenable, and the equivalence relation defined by the action is amenable by transitivity \cite{amrelation}.

Let us consider the homomorphism $\eta: G(F, F') \to \Z/2\Z$ given by the parity of the displacement length of some vertex of $T$. If we denote by $G^+$ the kernel of $\eta$, we have that $G^+$ acts without inversions.
Since the restriction map in bounded cohomology, i.e., \(\mbox{res}\colon H_b^n(G(F,F');\R)\to H_b^n(G^+;\R)^{\Z/2\Z}\), is injective \cite[Proposition 8.8.5]{Mon01} we can prove the theorem for $G=G^+$.

We verify that \(G^+\) satisfies the hypothesis of Theorem \ref{secondario}:\\ The action of \(G^+\) on \(T\) is still amenable as \(G^+\) is a closed subgroup of a group acting amenably \cite[Lemma 5.4.3]{Mon01}. Moreover, any segment of finite length can be sent to a fixed geodesic line \(L\) by means of an element in \(G^+\) as argued above.\\
Now, let us prove that there exists a line $L$ such that the set-wise stabilizer \(H\) of $L$ in \(G^+\) acts transitively on the geometric edges of $L$.\\
 Let $\tau=(n_1\ldots n_k)\ldots(m_1 \ldots m_l)$ be a permutation in $F$ written as a product of disjoint cycles, and let us suppose that $k \geq2$ ($F$ is non-trivial, so such a permutation exists). For any vertex \(v\in T\) we denote with \(E(v)\) the set of edges having \(v\) as one of the endpoints. 
 Let $v_0$ be any vertex of $T$. We construct inductively a line $L$ passing through $v_0$ by choosing edges in the following way:
 we choose the edges \(e_{n_1}\) and \(e_{n_2}\) \(\in E(v_0)\) labeled with \(n_1\) and \(n_2\) respectively. Let \(v_{-1}\) and \(v_1\in V(T)\) be such that \(e_{n_1}=(v_{-1},v_0)\) and \(e_{n_2}=(v_0,v_1)\). We want to extend the geodesic segment \([v_{-1},v_1]\) on both directions by choosing the labeling of the edges (Figure \ref{figura1}).

\begin{figure}
\centering
\[\begin{tikzcd}[column sep=7.5pt,row sep=16pt]
	& \cdots & \bullet & \cdots & \bullet & \bullet & \cdots & \bullet & \bullet & \cdots & \bullet & \cdots \\
	\\
	\cdots && {v_{-2}} && {v_{-1}} && {\textcolor{orange}{v_0}} && {v_1} && {v_2} && \cdots \\
	\\
	& \bullet & \cdots & \bullet & \cdots & \bullet & \cdots & \bullet & \cdots & \bullet & \cdots
	\arrow["{n_1}"', no head, from=3-1, to=3-3]
	\arrow["{n_3}", shift left=3, no head, from=3-3, to=1-2]
	\arrow["{n_{k-1}}"'{pos=0.8}, no head, from=3-3, to=1-2]
	\arrow["{n_k}"', no head, from=3-3, to=1-3]
	\arrow["{n_2}"', no head, from=3-3, to=3-5]
	\arrow[dashed, no head, from=3-3, to=5-2]
	\arrow["{n_3}", shift left=3, no head, from=3-5, to=1-4]
	\arrow["{n_{k-1}}"'{pos=0.8}, no head, from=3-5, to=1-4]
	\arrow["{n_k}"', no head, from=3-5, to=1-5]
	\arrow["{n_1}"', no head, from=3-5, to=3-7]
	\arrow[dashed, no head, from=3-5, to=5-4]
	\arrow["{n_3}", no head, from=3-7, to=1-6]
	\arrow["{n_k}"', no head, from=3-7, to=1-8]
	\arrow["{n_2}"', no head, from=3-7, to=3-9]
	\arrow[dashed, no head, from=3-7, to=5-6]
	\arrow["{n_1}", no head, from=3-9, to=1-9]
	\arrow["{n_4}"', shift right=3, no head, from=3-9, to=1-10]
	\arrow["{n_k}"{pos=0.8}, no head, from=3-9, to=1-10]
	\arrow["{n_3}"', no head, from=3-9, to=3-11]
	\arrow[dashed, no head, from=3-9, to=5-8]
	\arrow["{n_1}", no head, from=3-11, to=1-11]
	\arrow["{n_k}"{pos=0.8}, no head, from=3-11, to=1-12]
	\arrow["{n_4}"', shift right=3, no head, from=3-11, to=1-12]
	\arrow["{n_2}"', no head, from=3-11, to=3-13]
	\arrow[dashed, no head, from=3-11, to=5-10]
\end{tikzcd}\]

\caption{Sketch of the argument in the proof of Theorem \ref{principale}, item (1). The horizontal line represent the geodesic line obtained by choosing the coloring of the edges as described.}
    \label{figura1}
\end{figure}
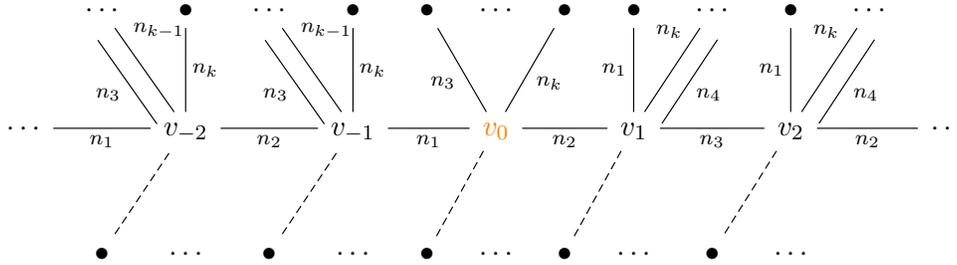

 If \(k=2\), we extend the segment by alternating edges labeled with \(n_1\) and \(n_2\) in both directions.

 If \(k\geq 3\), we extend the segment by choosing edges labeled with \(n_1\) and \(n_2\) alternatively in the direction of \(v_{-1}\), and edges labeled with \(n_2\) and \(n_3\) alternatively in the direction of \(v_1\) (this is the case displayed in Figure \ref{figura1}).

 Let $H$ be the set-wise stabilizer of the geodesic line \(L\) obtained in this way. 
 \begin{claim}
 The subgroup $H$ contains a translation along $L$ with displacement length equal to 2 and the rotation around the vertex \(v_0\).    
 \end{claim}
 \begin{proof}[Proof of the claim]
  The translation $t$ with displacement length equal to 2 can be achieved in the following way: let us set $t(v_i)=v_{i+2}$ and $\sigma(t,v)=\id$ for every vertex $v\in L$ except for $v_0$ and $v_{-1}$. The local action of $t$ around the vertices $v_{-1}$ and $v_0$ can be prescribed by an element of $F'$ such that $t(v_{-1})=v_1$ and such that $L$ is preserved by using the fact that $F'$ is 2-transitive. Hence, we can extend $t$ to an automorphism of $T$ (which by construction preserves $L$) thanks to the transitivity of \(F\): in fact, given a vertex $w$ at distance 1 from $L$, if we look at the local action around $w$, we have that only the image of the edge which connects $w$ with $L$ is already prescribed. The transitivity of $F$ grants that there exists a permutation in $F$ which sends $w$ to its prescribed image. We can iterate this procedure to the whole tree.
  
  The rotation $r$ of \(L\) around the vertex $v_0$ sends the vertices \(v_i\) to \(v_{-i}\) and fixes $v_0$. We prescribe that $\sigma(r,v_0)\in F'$ and $r$ exchanges the edges of $E(v_0)$ labeled with $n_1$ and $n_2$ (this can be achieved again using that $F'$ is 2-transitive). For $i\neq0$, we can prescribe that $\sigma(r, v_{\pm i})=\tau^{\pm i}$, which is an element of $F$. Once again, we can extend this element to an automorphism of the whole tree by transitivity of \(F\).  
 \end{proof}

By combining $t$ and $r$, it is straightforward to check that $H$ acts transitively on the geometric edges of $L$. Hence, we can apply Theorem \ref{bcsegmfinleng} to deduce that $G(F, F')^+$ is boundedly acyclic, and this concludes the proof.
\end{proof}

Combining this result with Remark \ref{notamen} and Lemma \ref{directprod}, we get that the groups $G(F, F')$ with $F$ such that the stabilizers of two points are trivial and $F'$ 2-transitive are non-amenable boundedly acyclic groups with the additonal property that every finitely generated subgroup which is a direct product has finite cardinality. 

\begin{proof}[Proof of Theorem \ref{principale}, item (2)]
Let $d\in \N_{\geq3}$. Then $G=G(F, F')$ acts on $T=\mathcal{T}_d$ transitively hence minimally.\\
First, we prove that $G(F, F')$ does not fix a point in $\partial T$, regardless of whether $F'$ is 2-transitive or not. We notice that it suffices to prove that the Burger-Mozes group $U(\id)$ does not fix a point in the boundary of $T$ as it is always a subgroup of \(G(F,F')\). 
Let us suppose by contradiction that $U(\id)$ fixes a point in $\partial T$. Then, there exists a geodesic half-ray $\gamma$ such that for each $g \in U(\id)$ we have that $g \cdot \gamma$ is at bounded distance from $\gamma$ (i.e., it is definitely contained in the support of $\gamma$). Then, there exist two vertices $v_1, v_2 \in \gamma$ such that the coloring of the edges of $\gamma$ around $v_1$ and $v_2$ is the same. Since the degree $d$ of $T$ is at least 3, there is at least a vertex $w$ which is adjacent to $v_2$ but does not lie in $\gamma$. 
The action of $U(\id)$ on $T$ is vertex-transitive, so there is an element $g \in U(\id)$ such that $g \cdot v_1=w$. We claim that the geodesic half-ray $g \cdot \gamma$ is not in the same equivalence class of $\gamma$. If it was the case there should exist $t_0 \in \R$ such that $g \cdot \gamma(t)\in \gamma$ for each $t\geq t_0$. Let $e$ be the edge between $w$ and $v_2$. Since $e$ has a different color than the edges around $v_1$ which lie in $\gamma$ (let them be $e_1, e_2$), it follows that an element of $U(\id)$ cannot send $e_i$ to $e$ for $i=1,2$. The conclusion follows from the fact that $T$ is a tree, so it has no cycles (see also Figure \ref{figura2}).

\begin{figure}
\[\begin{tikzcd}[column sep=small,row sep=tiny]
	&&&&&&&& {g(v_1)} \\
	{} && \bullet && {} && {} && w \\
	{} & {} && {} & {} && {} & {} &&& {} \\
	{} & {} && {} & {} & {} & {} && {} & {} & {} \\
	\bullet && {v_1} & {} & \bullet && \bullet && {v_2} && \bullet & {} & {} & {} & {} \\
	&&&&&& {}
	\arrow[shorten >=14pt, dashed, no head, from=2-3, to=2-1]
	\arrow[shorten <=14pt, dashed, no head, from=2-5, to=2-3]
	\arrow[Rightarrow, no head, from=2-9, to=1-9]
	\arrow[shorten >=15pt, dashed, no head, from=2-9, to=2-7]
	\arrow["{g\cdot \gamma}", curve={height=-18pt}, dashed, no head, from=2-9, to=5-14]
	\arrow[dashed, no head, from=5-1, to=3-1]
	\arrow["1", no head, from=5-1, to=5-3]
	\arrow["3", no head, from=5-3, to=2-3]
	\arrow["2", no head, from=5-3, to=5-5]
	\arrow[dashed, no head, from=5-5, to=3-5]
	\arrow[dashed, no head, from=5-7, to=3-7]
	\arrow[dashed, no head, from=5-7, to=5-5]
	\arrow["1", no head, from=5-7, to=5-9]
	\arrow["3", no head, from=5-9, to=2-9]
	\arrow["2", no head, from=5-9, to=5-11]
	\arrow[dashed, no head, from=5-11, to=3-11]
	\arrow[""{name=0, anchor=center, inner sep=0}, shorten >=11pt, dashed, no head, from=5-11, to=5-15]
	\arrow["\underbrace{\hspace{9cm}}_{\gamma}"', shift right=2, curve={height=2pt}, draw=none, from=5-1, to=0]
\end{tikzcd}\]
\caption{Sketch of the argument in the proof of Theorem \ref{principale}, item (2), with $d=3$. The half-ray $\gamma$ is the horizontal half-line which proceeds to the right of the picture.}
    \label{figura2}
\end{figure}
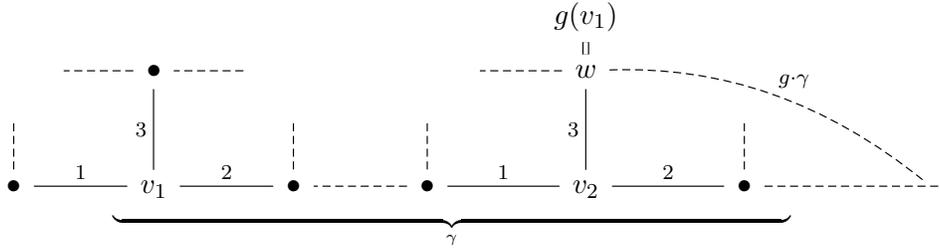

Hence \(G\) does not satisfy statement \((a)\) of Theorem \ref{thmIozzi}.
In order to conclude, it suffices to prove that if $F'$ is not 2-transitive then $G$ does not satisfy statement $(b)$ of Theorem \ref{thmIozzi} and hence, it has infinite dimensional second bounded cohomology.

Let us denote by $\mathcal{E}^n_v$ the set of geodesic segments of length $n$ starting at the vertex $v$. We distinguish two cases: if $F$ is not transitive, then neither is $F'$ (since it has to preserve the orbits of $F$). This implies that, for each vertex $v \in T$, $G$ does not act transitively on $\mathcal{E}^1_v$, i.e., the geodesic segments of length 1 (i.e., the edges) having $v$ as one of the endpoints.
If $F$ is transitive, then $F'$ is transitive too, but not 2-transitive by assumption. By Lemma \ref{stabiltrans}, this implies that, given a vertex $v$ and an edge $e\in \mathcal{E}^1_v$, the stabilizer of $e$ does not act transitively on $\mathcal{E}^1_v \setminus \{e\}$. In other words, there exist two edges $e_1, e_2 \in \mathcal{E}^1_v \setminus \{e\}$ with $e_1\neq e_2$ such that no element in the stabilizer of $e$ maps $e_1$ to $e_2$. 
Let $x$ be the other endpoint of $e$, and let us denote by $G_x$ the stabilizer of $x$. We consider the action of $G_x$ on $\mathcal{E}^2_x$, i.e., the geodesic segments of length 2 having $x$ as one of the endpoints. Let us take the geodesic segments $\gamma_1$ and $\gamma_2$ of length 2 such that $\gamma_i$ is obtained by the concatenation $e_i*e$. By definition, $\gamma_i \in \mathcal{E}^2_x$ for each $i=1,2$. Now, if $G$ acted transitively on $\mathcal{E}^2_x$ there would exist $g \in G$ such that $g \cdot \gamma_1=\gamma_2$. 
This implies that there should exist $\sigma$ in the stabilizer of $e$ such that $\sigma(e_1)=e_2$, but this contradicts the hypotheses made on \(e_1\) and \(e_2\). So $G$ does not act transitively on $\mathcal{E}^2_x$. Hence \(G\) does not satisfy statement \((b)\) of Theorem \ref{thmIozzi}, from which we get the conclusion.
\end{proof}

\bibliography{Bibliografia}
\bibliographystyle{alpha}

\end{document}